\documentclass[12pt]{amsart}
\usepackage{epsfig,color}
\usepackage{blindtext}
\usepackage{graphicx}
\usepackage{enumitem}
\usepackage{url}
\usepackage{amssymb}
\usepackage{graphicx,import}
\usepackage{comment}
\usepackage{esint}
\usepackage{xcolor}
\usepackage{mathtools}
\usepackage{comment}
\usepackage{tikz}
\usepackage{stmaryrd} 

\allowdisplaybreaks

\usepackage[margin = 1in] {geometry}

\usepackage{hyperref}
\usepackage{dsfont}

\newtheorem{theorem}{Theorem}[section]

\newtheorem{prop}[theorem]{Proposition}
\newtheorem{lemma}[theorem]{Lemma}

\theoremstyle{definition}

\newtheorem*{remark*}{Remark}

\newtheorem{remark}[theorem]{Remark}

\theoremstyle{claim}

\numberwithin{equation}{section}

\newcommand{\BC}{\mathbf{C}}
\newcommand{\CE}{\mathcal{E}}

\renewcommand{\H}{\mathcal{H}}

\newcommand{\R}{\mathbb{R}}

\newcommand{\N}{\mathbb{N}}
\newcommand{\Z}{\mathbb{Z}}

\newcommand{\eps}{\varepsilon}

\newcommand{\sing}{\textnormal{sing}}
\newcommand{\weakly}{\rightharpoonup}
\newcommand{\del}{\partial}
\newcommand{\graph}{\textnormal{graph}}
\newcommand{\res}{\mathbin{\hspace{0.1em}\vrule height 1.3ex depth 0pt width 0.13ex\vrule height 0.13ex depth 0pt width 1.0ex}}

\title{Uniqueness of regular tangent cones for immersed stable hypersurfaces}

\author{Nick Edelen}
\address{Department of Mathematics, University of Notre Dame, Notre Dame, IN, 46556, USA}
\email{nedelen@nd.edu}

\author{Paul Minter}
\address{Department of Mathematics, Princeton University, Princeton, NJ 08540, USA; School of Mathematics, Institute for Advanced Study, 1 Einstein Dr., Princeton, NJ 08540, USA}
\email{pm6978@princeton.edu, pminter@ias.edu}

\begin{document}

\maketitle

\begin{abstract}
    We establish uniqueness and regularity results for tangent cones (at a point or at infinity) with isolated singularities arising from a given immersed stable minimal hypersurface with suitably small (non-immersed) singular set. In particular, our results allow the tangent cone to occur with any integer multiplicity.
\end{abstract}

\section{Introduction}

Let $M^n$ be a minimal hypersurface in Euclidean space with $0\in \sing(M) \equiv \overline{M}\setminus M$ a singular point of $M$. One may analyze the first-order singular behaviour of $M$ at $0$ by ``blowing-up'' $M$ at $0$. Indeed, a well-known consequence of the monotonicity of the mass ratio $r\mapsto r^{-n}\H^n(M\cap B_r(0))$ is that for any sequence of radii $r_i\downarrow 0$ there is a subsequence $r_{i^\prime}$ such that, as varifolds, $r_{i^\prime}^{-1} M\weakly \BC$. Such a varifold limit $\BC$ is a called a \emph{tangent cone} to $M$ at $0$, with the conical structure again due to the monotonicity formula for the mass ratio. One hopes to deduce properties of the singularity in $M$ from the tangent cone $\BC$, with a key issue being whether $\BC$ is unique (namely, independent of the sequence $r_i$ and the subsequent subsequence).

If the mass ratio is bounded at infinity, i.e. 
$\limsup_{r\to\infty}r^{-n}\H^n(M\cap B_r(0))<\infty,$
one can analogously consider ``blowing-down'' $M$. (Such an assumption holds when $M$ is area-minimizing, for instance.) In this setting, one takes a sequence of radii $r_i\uparrow \infty$, and the mass ratio upper bound allows us to extract limits $r_{i^\prime}^{-1} M\weakly \BC$. In this context $\BC$ is called a \emph{tangent cone at infinity} or a \emph{blow-down cone}. One wishes to deduce asymptotic properties of $M$ from $\BC$, with again uniqueness of $\BC$ a key issue. We note that analysis of blow-down cones was a key step in the resolution of the Bernstein problem for minimal graphs.

Few general results in these directions are known. Many results require the cone $\BC$ to be multiplicity one in an appropriate fashion (e.g. \cite{All72, Sim83a, Sim93, BK17}). In the case of stable minimal hypersurfaces, some results are known in certain cases, typically with some restriction on the possible singularities (e.g. \cite{Wic08, Wic14, MW23, Bel23, Min24}). 

Here, we resolve this question for stable immersed minimal hypersurfaces (with a small singular set) when $\BC$ is a \emph{regular cone}, i.e. $\sing(\BC) = \{0\}$ (or equivalently the \emph{link} $\Sigma := \BC\cap \del B_1(0)$ is smoothly embedded). When $\BC$ is a regular cone we may therefore write $\BC = q|\BC^\prime|$, where $\BC^\prime$ is a smooth cone with $\overline{\BC^\prime}\setminus \BC^\prime = \{0\}$ and $q\in \Z_{\geq 1}$. When $\BC$ has multiplicity one (i.e. $q=1$), the corresponding result was established by Simon \cite{Sim83a} (Simon's result holds more generally for minimal surfaces in any dimension and codimension, with no stability assumption).

\begin{theorem}\label{thm:main-1}
    Let $M$ be an immersed stable minimal hypersurface in $B_1^{n+1}(0)\subset\R^{n+1}$ (resp. $\R^{n+1}\setminus B_1^{n+1}(0)$) with $\H^{n-2}(\sing(M))=0$. Suppose there is a sequence $r_i\to 0$ (resp. $r_i\to \infty$) such that $r_i^{-1}M\weakly q|\BC|$ as varifolds, where $\BC$ is a regular cone. Then, $r^{-1}M\to q|\BC|$ as $r\to 0$ (resp. $r\to\infty$), and moreover the convergence is a rate bounded by $C|\log(r)|^{-\gamma}$ for some $C = C(\BC,q)$ and $\gamma = \gamma(\BC,q)>0$. More precisely, there is a (smooth) $q$-valued function $u:\BC\cap B_{1/2}^{n+1}(0)\to \mathcal{A}_q(\BC^\perp)$ (resp. $u:\BC\setminus \overline{B}^{n+1}_2(0)\to \mathcal{A}_q(\BC^\perp)$) with $\textnormal{graph}(u) = M$, with $u$ having the claimed decay rate at $0$ (resp. $\infty$). 
\end{theorem}

\begin{remark*}
    Here $\sing(M)$ is the non-immersed singular set of $M$. If we instead assume that $M$ is embedded, Theorem \ref{thm:main-1} essentially follows from \cite{Sim83a} combined with the regularity theorem of Schoen--Simon \cite{SS81} (in fact by the strong maximum principle \cite{ilmanen}, in this case $M$ would necessarily be a single sheet with multiplicity $q$ near $0$).
\end{remark*}

Provided there is a suitable `immersed' regularity theorem near immersed singularities, we can also prove an analogous result when $\BC$ is an immersed regular cone (i.e. the link $\Sigma\subset S^n$ is smoothly immersed) with $0\in \BC$ a non-immersed singularity. For density $2$ immersed singularities, the required regularity theorem is known in the more general setting where $M$ is simply a stationary integral $n$-varifold with stable regular part and no triple junction singularities (see \cite[Theorem D]{MW23}).

\begin{theorem}\label{thm:main-2}
    Let $M$ be a stationary integral varifold in $B^{n+1}_1(0)\subset\R^{n+1}$ (resp. $\R^{n+1}\setminus B^{n+1}_1(0))$ which has stable regular part and no triple junction singularities. Suppose there is a sequence $r_i\to 0$ (resp. $r_i\to\infty$) such that $r_i^{-1}M\weakly |\BC|$ as varifolds, where $\BC$ is a (multiplicity one) immersed regular cone, having the property that all immersed singularities in $\BC$ have density $2$. Then, $r^{-1}M\to |\BC|$ as $r\to 0$ (resp. $r\to\infty$), and moreover the convergence is a rate bounded by $C|\log(r)|^{-\gamma}$ for some $C = C(\BC)$,$\gamma = \gamma(\BC)>0$.
\end{theorem}
\begin{remark*}
    Theorem \ref{thm:main-2} includes, for example, the case where $\BC$ is the sum of two transverse copies of a given (multiplicity one) regular cone (e.g. Simons' cone). In this case, $M$ itself splits into the union of two minimal graphs, one over each copy of the cone.
\end{remark*}

In the context of Theorem \ref{thm:main-1}, one could ask whether the $q$-valued function globally splits into $q$ single-valued functions, or whether it is possible to have objects which resemble ``branched'' regular cones. It is always the case the $u$ will globally split into $q$ single-valued functions when the link $\Sigma$ is simply connected (c.f. \cite[Theorem 8.3]{SW16}). However, when it is not (e.g. $\Sigma = S^1\times S^m$ for suitably large $m$) then one can adapt the methods of Caffarelli--Hardt--Simon \cite{CHS} and Simon \cite{Sim85} to our framework in order to construct examples of genuinely immersed (but not embedded) minimal hypersurfaces with isolated singularities.

\begin{theorem}\label{thm:main-3}
    Let $\BC$ be a regular hypercone such that its link $\Sigma$ is not simply-connected. Then, there exists $A = A(\Sigma)\subset \{2,3,\dotsc\}$, with $\#A\geq 1$ such that for each $q\in A$, there exists an immersed (multiplicity one) minimal hypersurface $M\subset B^{n+1}_1(0)$ with $(\overline{M}\setminus M)\cap B^{n+1}_1(0) = \{0\}$, and with $q|\BC|$ being the (unique) tangent cone to $M$ at $0$, and moreover (in the notation of Theorem \ref{thm:main-1}) for each $r>0$ the function $u$ restricted to $\BC\cap \del B_r^{n+1}(0)$ cannot be written as a sum of $q$ single-valued $C^2$ functions. Moreover, if $\BC$ is strictly stable, then $M$ is stable as an immersion.
\end{theorem}

\begin{remark*}
    The restriction on the possible values of $q$ in Theorem \ref{thm:main-3} comes from the restriction on the possible degrees of covering spaces of $\Sigma$. Our examples will be induced by a single choice of cover, and so the specific cover does not change with the radius.
\end{remark*}

\begin{remark*}
    Any cone with link of the form $\Sigma = S^1\times\prod_{i=1}^p S^{n_i}$ for integers $p,n_i\geq 1$ and $\sum_{i=1}^p n_i \geq 6$ is area-minimising (\cite[Theorem 5.1.1]{Law91}) and so is strictly stable. Therefore, such $\BC$ give rise to genuinely immersed stable minimal hypersurfaces which have regular tangent cones of higher multiplicity (and for any $q\geq 2$ in the notation of Theorem \ref{thm:main-3}).
\end{remark*}

\textbf{Overview of the Proof.} We prove Theorem \ref{thm:main-1} by adapting the Łojasiewicz--Simon-type decay estimates from \cite{Sim83a} to the immersed setting, using the recent sheeting theorem for stable minimal immersions due to Bellettini \cite[Theorem 6]{Bel23}. Our basic idea is to reinterpret each slice $M\cap \del B_r$, originally being a $q$-valued graph over $\BC\cap \del B_r$, as instead a single-valued graph over some $q$-cover of $\BC\cap \del B_r\cong \Sigma$ (which cover depends on the $q$-valued graph over $\BC\cap \del B_r$). We can then apply the now-standard argument of \cite{Sim83a} to obtain a Łojasiewicz--Simon inequality for single-valued graphs over each $q$-cover of $\Sigma$. Since there are only finitely many possible $q$-covers (up to diffeomorphism) we then deduce a Łojasiewicz--Simon inequality for $q$-valued functions over $\Sigma$.

The main subtlety, which is the reason that the required estimates do not immediately follow from \cite{Sim83a} after passing to a suitable cover, is that covers do not always respect the cylindrical structure. In other words, if $T$ is a $q$-cover of $\Sigma\times [a,b]$, then $T$ \emph{need not} have the form $\Sigma^\prime\times [a,b]$ for some $q$-cover $\Sigma^\prime$ of $\Sigma$, i.e. the cover could change from slice-to-slice.
Instead, we will only appeal to the $q$-covers of $\Sigma$ to establish a suitable Łojasiewicz--Simon inequality, and prove the decay at the level of $q$-valued graphs directly. We adapt the ``simpler'' decay approach of \cite{Sim96} to multi-valued functions, and (perhaps interesting in its own right) also extend the approach in \cite{Sim96} to one which is ``symmetric'' in the radius, thereby giving not only uniqueness of tangent cones both at $0$ and $\infty$, but also an a priori Dini-type estimate, like in \cite{Ede21} and \cite{MN22}.

Indeed, our main estimate is the following. See Section \ref{sec:prelim} for details on the notation of $q$-valued functions.

\begin{theorem}[Dini-type a priori estimate]\label{thm:dini}
    Let $\BC$ be a minimal regular hypercone in $\R^{n+1}$ with link $\Sigma$, and let $q\in \Z_{\geq 1}$, $\beta > 0$. Then, there are $\eps = \eps(\Sigma,q)\in (0,1)$, $\alpha = \alpha(\Sigma,q)\in (0,1]$, and $C = C(\Sigma,q, \beta)\in (0,\infty)$ such that the following holds.

    Fix $0<\rho<R$. Let $g$ be a $C^2$ metric on $A_{R,\rho}:= B_R\setminus\overline{B}_\rho$, and let $M$ be an immersed minimal hypersurface in $A_{R,\rho}$ with respect to $g$ which is the graph over $\BC$ be a $C^3$ $q$-valued function $u$. Suppose $g$ and $u$ satisfy
    \begin{equation}\label{E:Dini-1}
        \sum^3_{k=0}|x|^{k-1}|\nabla^k u| < \eps, \qquad \sum^2_{k=0}|x|^k|D^k(g-g_{\textnormal{Eucl}})|\leq \delta_-|x|^{-\beta} + \delta_+|x|^\beta
    \end{equation}
    for some $\delta_+,\delta_- \geq 0$ satisfying $\delta_-\rho^{-\beta} + \delta_+R^\beta \leq \eps$.
Then we have the estimate
    \begin{align}
        \label{E:Dini-2}\sup_{r,s\in(\rho,R)}\left\|\mathcal{G}\left(\frac{u(r)}{r},\frac{u(s)}{s}\right)\right\|_{L^2(\Sigma)} & \leq C\int^R_\rho\left(\int_\Sigma \left|r\frac{\del}{\del r}\left(\frac{u(r\theta)}{r}\right)\right|^2 d\theta\right)^{1/2}\frac{dr}{r}\\
        & \label{E:Dini-3}\leq C^2\left(\int_{M\cap A_{R,\rho}}\frac{|\pi^\perp_M(X)|^2}{|X|^{n+2}}\ d\H^n(X) + \delta_-^2\rho^{-2\beta} + \delta_+^2R^{2\beta}\right)^{\alpha/2}.
    \end{align}
    In particular, if $M$ is minimal in $B_R(0)$ and $g = g_{\textnormal{Eucl}}$, then
    \begin{equation}
        \label{E:Dini-4}\int^R_\rho\left(\int_\Sigma \left|r\frac{\del}{\del r}\left(\frac{u(r\theta)}{r}\right)\right|^2 d\theta\right)^{1/2}\frac{dr}{r} \leq C\left(\Theta_M(R) - \Theta_M(\rho)\right)^{\alpha/2}.
    \end{equation}
\end{theorem}

\begin{remark}
    If $\BC$ is such that $\alpha = 1$ in Theorem \ref{thm:dini}, in turn one gets a better decay rate (namely, $C^{1,\gamma}$ for some $\gamma = \gamma(\Sigma,q)\in (0,1)$) as a conclusion in Theorem \ref{thm:main-1} and Theorem \ref{thm:main-2}. This corresponds to the case when $\BC$ is ``integrable''.
\end{remark}

Theorem \ref{thm:main-2} is essentially a direct consequence of Simon's original argument, just using the relevant transverse $\eps$-regularity theorem. This enables one to ``separate'' the two functions at each self-intersection in $\BC$, reducing to the multiplicity one case. In the special case where $\BC$ is the sum of two transverse multiplicity one cones, the $\eps$-regularity theorem (which here is \cite[Theorem D]{MW23}) truly allows one to re-run Simon's original argument, essentially treating each cone separately. (This is similar to how in Theorem \ref{thm:main-1}, if one knows the link is simply-connected then the proof is much simpler as one can reduce to single-valued functions and re-run the $q=1$ argument of \cite{Sim83a} to each.)

Regarding Theorem \ref{thm:main-3}, again the proof closely follows those seen in \cite{CHS, Sim85}, and is similar to the multiplicity one case (as one just needs to lift to a single cover and apply the arguments there). It would be interesting to know whether the examples constructed in Theorem \ref{thm:main-3} arise can arise as limits of smoothly immersed minimal hypersurfaces, similar to the Hardt--Simon foliation (c.f. \cite{HS85}).

\subsection{Organization}
In Section \ref{sec:prelim} we review the notation and establish some preliminary properties of multi-valued functions, as well as establish the Łojasiewicz--Simon inequality for multi-valued functions. In Section \ref{sec:estimates}, we prove the main decay result which will lead to our Dini-type estimate. In Section \ref{sec:proofs-1} we prove Theorem \ref{thm:dini} and Theorem \ref{thm:main-1}, in Section \ref{sec:proofs-2} we prove Theorem \ref{thm:main-2}, and in Section \ref{sec:proofs-3} we prove Theorem \ref{thm:main-3}. 

\subsection{Acknowledgements}
N.E. was supported by NSF grant DMS-2204301. This research was conducted during the period P.M. served as a Clay Research Fellow.

\section{Preliminaries and Notation}\label{sec:prelim}

We work in $\R^{n+1}$. Write $B_r$ for the open ball of radius $r$ centered at $0$ in $\R^{n+1}$, and write $A_{R,\rho}:= B_R\setminus \overline{B}_\rho$ for the open annulus centered at $0$ between radii $R$ and $\rho$ ($R>\rho$). Throughout this paper, $C$ will denote a generic constant $\geq 1$ which may increase from line to line. We use the notation $o(t)$ to indicate a quantity which $\to 0$ as $t\to 0$.

For an immersed hypersurface $M^n\subset B_R$ and $r<R$ we define the \emph{mass ratio} (at $0\in\R^n$) by
$$\Theta_M(r):= \frac{\H^n(M\cap B_r)}{\omega_n r^n}$$
where $\omega_n = \H^n(B_1^n(0))$. When $M$ is minimal, $\Theta_M(r)$ is non-decreasing in $r$ and so $\Theta_M(0):=\lim_{r\downarrow 0}\Theta_M(r)$ exists. For $x\in M$ we write $\left.\pi_M\right|_x$ for the orthogonal projection onto the tangent plane $T_xM$, and $\left.\pi^\perp_M\right|_x$ for the orthogonal projection onto $T_x^\perp M$.

Unless specified otherwise, for the duration of this paper we fix $\Sigma$ a smooth, closed Riemannian $(n-1)$-manifold in $\R^{n+1}$, and write $\nabla$ for its connection derivative (associated to the induced Levi--Civita connection). Given a function $u:\Sigma\times I\to \R$ for some interval $I$, we write $u(t)(x):= u(x,t)$ and $\dot{u}:= \frac{d}{dt}u$, where $(x,t)\in \Sigma \times I$. As in \cite{Sim83a}, we use the norms
\begin{equation}
    |u(t)|_k := |u(t)|_{C^k(\Sigma)}, \qquad |u(t)|^*_k := \sum^k_{j=0}|D^j_t u(t)|_{C^{k-j}(\Sigma)}.
\end{equation}
Let $\CE$ be a functional defined on $C^1(\Sigma)$ of the form
\begin{equation}\label{E:E}
    \mathcal{E}(v) = \int_\Sigma E(x,v,\nabla v)\ dx.
\end{equation}
We assume the function $E(x,p,q)$ is a smooth function on $\Sigma\times \R\times T\Sigma \ni (x,p,q)$ which satisfies the convexity condition
\begin{equation}\label{E:convexity}
    \left.\frac{d^2}{ds^2}\right|_{s=0} E(x,0,sp)\geq c|p|^2
\end{equation}
for some constant $c$ independent of $x$ and $p$. Moreover, as in \cite{Sim83a}, we assume $E(x,p,q)$ depends analytically on $p$ and $q$.

Write $\mathcal{M}$ for the negative $L^2(\Sigma)$-gradient of $\CE$, in the sense that
$$\int_\Sigma \mathcal{M}(v)\zeta = -\left.\frac{d}{ds}\right|_{s=0}\mathcal{E}(v+s\zeta) \qquad \text{for all }v,\zeta\in C^2(\Sigma).$$
The convexity assumption \eqref{E:convexity} implies that $\mathcal{M}$ is a quasi-linear, second-order, elliptic operator on $C^2(\Sigma)$, provided $|u|_{C^1(\Sigma)}$ is sufficiently small. We shall assume that $\mathcal{M}(0) = 0$ (i.e. that $v\equiv 0$ is a critical point of $\mathcal{E}$). We write $L$ for the linearization of $\mathcal{M}$ at $0$, i.e.
$$L(\zeta)= \left.\frac{d}{ds}\right|_{s=0}\mathcal{M}(s\zeta).$$
$L$ is a linear, second-order, self-adjoint elliptic operator on $\Sigma$. We refer the reader to \cite{Sim83a, Sim85} for more details.

We shall need the following basic topological fact.

\begin{lemma}\label{lemma:covers}
    Let $\Sigma$ be a closed Riemannian $n$-manifold and $q\in\Z_{\geq 1}$. Then, up to covering diffeomorphism, there are at most $N = N(\Sigma,q)$ (possibly disconnected) $q$-covers of $M$.
\end{lemma}

\begin{proof}
    We sketch the proof for completeness. Since $M$ is a smooth closed manifold, $\pi_1(\Sigma)$ is finitely generated (in fact, it is finitely presented); this can be seen by noting that $\Sigma$ is homotopically equivalent to some CW-complex. Moreover, covering maps of degree $q$ are, by the Galois correspondence, in bijection with the index $q$ subgroups of $\pi_1(\Sigma)$. Each such subgroup $G$ defines a homomorphism $\pi_1(\Sigma)\to \text{Sym}(\pi_1(\Sigma)/G)\cong S_q$ (via left-action on $\pi_1(\Sigma)/G$ by $\pi_1(\Sigma)$). As $\pi_1(\Sigma)$ is finitely generated, such a homomorphism is determined by where the $m$ generators of $\pi_1(\Sigma)$ are sent, and so there are at most $(q!)^m$ such homomorphisms. Since distinct index $q$ subgroups give rise to distinct homomorphisms, this gives that there are at most $(q!)^m$ such subgroups, and hence at most $(q!)^m$ covers of degree $q$.
\end{proof}

\subsection{Multi-Valued Functions}

We write $\mathcal{A}_q(\R)$ for the set of unordered $q$-tuples of real numbers. For $a,b\in \mathcal{A}_q(\R)$, write $\mathcal{G}(a,b) = \inf\sum_i |a_i-b_i|^2$, where the infimum is over all possible ways of decomposing $a = \sum_i \llbracket a_i\rrbracket$, $b = \sum_i \llbracket b_i\rrbracket$.

In this work, we will only be interested in multi-valued functions which do \emph{not} have branch points. As a result, some definitions related to multi-valued functions can be simplified. Indeed, given a smooth manifold $\Sigma$ and $k\in \Z_{\geq 0}$, we say a function $u:\Sigma\to \mathcal{A}_q(\R)$ is (\emph{non-branched}) $C^k$\emph{-regular} if near any point $x\in \Sigma$, there is a neighbourhood $U\ni x$ on which we may write
$$u = \sum_i \llbracket u_i\rrbracket$$
where $u_i:U\to \R$ is a $C^k$ single-valued function for each $i=1,\dotsc,q$. To avoid possible notational confusions with usual (possibly branched) multi-valued functions, we will write $C^k_q(\Sigma)$ for the space of (non-branched) $C^k$ $q$-valued functions. Given $u,v\in C^k_q(\Sigma)$, we write
$$\mathcal{G}(u,v)|_x := \mathcal{G}(u(x),v(x)) \qquad \text{for }x\in \Sigma.$$
For $u\in C^k_q(\Sigma)$, we define the \emph{coincidence set} 
$$\mathcal{K}_u:= \{x\in \Sigma: (u_i(x),\nabla u_i(x)) = (u_j(x),\nabla u_j(x))\text{ for some }i\neq j\}.$$
\begin{remark}
    For $u\in C^k_q(\Sigma)$ and $p\in \Sigma$, regardless of whether the decomposition into $u_i$ is unique, the set
    $$\{(u_i(p),\nabla u_i(p),\dotsc,\nabla^k u_i(p))\}_{i=1}^q$$
    is well-defined independent of the choice of decomposition.
\end{remark}
We define $|u|_{C^k(\Sigma)} := \sum_i |u_i|_{C^k(\Sigma)}$, and if $u\in C^k_q(\Sigma\times I)$ then $|u|_k = \sum_i |u_i|_k$ and $|u|_k^* = \sum_i |u_i|^*_k$; again, these definitions are independent of the choice of decomposition $u_i$. Similarly, we set
$$\|u\|_{L^2(\Sigma)}:= \left(\int_\Sigma\sum_i |u_i|^2 dx\right)^{1/2}.$$
We will often use the shorthand $|u(x)|^2:= \sum_i|u_i(x)|^2$.

If $F$ is a differentiable operator $C^k(\Sigma)\to C^\ell(\Sigma)$, then for $u\in C^k_q(\Sigma)$ and $f\in C^\ell_q(\Sigma)$, we interpret the equation $Fu = f$ to mean for every $x\in \Sigma$, there is a decomposition $u = \sum_i\llbracket u_i\rrbracket$ and $f = \sum_i\llbracket f_i\rrbracket$ on some neighbourhood of $x$ so that $Fu_i = f_i$ on $U$ for each $i=1,\dotsc,q$.

We say $u\in C^k_q(\Sigma)$ is \emph{uniquely decomposable} if locally about every point in $\Sigma$ the decomposition of $u$ is unique up to ordering, i.e. if for each $x\in \Sigma$, there is a neighbourhood $U$ of $x$ such that if $u = \sum_i \llbracket u_i\rrbracket = \sum_i \llbracket \widetilde{u}_i\rrbracket$ for $u_i,\widetilde{u}_i\in C^k(\Sigma)$, then $\widetilde{u}_i = u_{\sigma(i)}$ for some permutation $\sigma\in S_q$.

Secretly, the $q$-valued functions we are interested in are uniquely decomposable (this is not true for general $q$-valued functions). This follows from the following lemma, using that fact that the coincidence set for two minimal graphs always has codimension at least $2$.

\begin{lemma}\label{lemma:ud}
    If $u\in C^k_q(\Sigma)$ for some $k\geq 1$, and $\textnormal{codim}_\H(\mathcal{K}_u)>1$, then $u$ is uniquely decomposable.
\end{lemma}
\begin{proof}
    It suffices to prove the case $k=1$ as this implies the rest. The case $k=1$ however is immediate from the continuity properties of $u_i$ and their derivatives along with the fact that (each connected component of) $\Sigma\setminus \mathcal{K}_u$ is connected (due to the size assumption on $\mathcal{K}_u$).
\end{proof}

The key properties of uniquely decomposable multi-valued $C^k_q$ functions we need are summarised in the following lemma.

\begin{lemma}\label{lemma:ud-properties}
    Fix $k\geq 1$ and suppose $u\in C^k_q(\Sigma)$ is uniquely decomposable. Then:
    \begin{enumerate}
        \item[\textnormal{(i)}] if $U\subset\Sigma$ is contractible, then $u|_U = \llbracket u_1\rrbracket + \cdots + \llbracket u_q\rrbracket$, where $u_i\in C^k(U)$ are single-valued;
        \item [\textnormal{(ii)}] the set $\{(x,u_i(x)):x\in \Sigma,\, i=1,\dotsc,q\}$ is a (possibly disconnected) smooth degree $q$ cover of $\Sigma$, with $C^k$ covering map $\pi(x,u_i(x)) = x$.
    \end{enumerate}
\end{lemma}

\begin{proof}
    First we prove (i). We may without loss of generality assume $U$ is connected, else we can work on each connected component individually. 
    
    Since $U$ is contractible, there exists an open simply connected set $U^\prime\subset \Sigma$ obeying $U\subset U^\prime$. Therefore, it suffices to prove the claim assuming $U$ is open and simply connected instead. Fix $x\in U$, and choose a neighbourhood $V\ni x$ for which on $V$ we have $u = \sum_i \llbracket u_i\rrbracket$, for some unique choice (up to reordering) of $u_i\in C^k(V)$. We extend the $u_i$ to all of $U$ as follows. Fix $y\in U$, and take a path $\gamma$ in $U$ joining $x$ to $y$. As $\gamma$ is compact we therefore can use the unique decomposability of $u$ to extend the definition of $u_i$ to $y$. All that remains is to show that this definition is independent of the choice of path $\gamma$. However this follows from the simply connectedness of $U$.

    Regarding (ii), the fact that the set is a $q$-cover follows from the definition. Moreover the regularity of the covering map follows simply because locally the $u_i$ are $C^k$ single-valued functions.
\end{proof}

The main purpose of Lemmas \ref{lemma:ud} and \ref{lemma:ud-properties} is to prove the following Łojasiewicz--Simon inequality for multi-valued functions. For shorthand, let us define $\mathcal{E}_q$ acting on $u\in C^1_q(\Sigma)$ by
$$\mathcal{E}_q(u):=\int_\Sigma \sum_i E(x,u_i(x),\nabla u_i(x))\ dx.$$
\begin{prop}\label{prop:ls}
    There exist constants $\alpha = \alpha(\Sigma,q)\in (0,1]$ and $\sigma = \sigma(\Sigma,q)>0$ such that if $u\in C^3_q(\Sigma)$ is uniquely decomposable and satisfies $|u|_{C^3}\leq \sigma$, then
    \begin{equation}\label{E:ls}
        |\mathcal{E}_q(u)-\mathcal{E}_q(0)|^{2-\alpha}\leq\int_\Sigma\sum^q_{i=1}|\mathcal{M}(u_i)|^2\ dx.
    \end{equation}
\end{prop}
\begin{proof}
    From Lemma \ref{lemma:ud-properties}, if we let
    $$\Sigma^\prime := \{(x,u_i(x)):x\in\Sigma_i, i=1,\dotsc,q\}$$
    then $\Sigma^\prime$ has the structure of a $C^3$ $(n-1)$-manifold, and the projection map $\pi:\Sigma^\prime\to \Sigma$ given by $\pi(x,u_i(x)) = x$ is a $C^3$ covering map of degree $q$. Endow $\Sigma^\prime$ with the pullback metric $g^\prime = \pi^*g_\Sigma$ so that $\pi$ is a local isometry. We remark that, although we use $u$ to define the topological structure of $\Sigma^\prime$, from Lemma \ref{lemma:covers}, the Riemannian manifold $\Sigma^\prime$ is one of at most $N(\Sigma,q)$ (possibly disconnected) locally-isometric degree $q$ covering spaces of $\Sigma$, and in particular secretly depends only on $\Sigma$.

    Define the energy functional $\mathcal{E}^\prime$ on $C^1(\Sigma^\prime)$ by
    $$\mathcal{E}^\prime(v):=\int_{\Sigma^\prime}E(\pi(x),v(x),D\pi|_x\nabla v(x))\ d\H^{n-1}(x), \qquad v\in C^1(\Sigma^\prime),$$
    where $E$ is as in \eqref{E:E}. Note that if $v$ is supported in a set $U$ on which $\pi|_U$ is a diffeomorphism, then $\mathcal{E}^\prime(v) = \mathcal{E}(v\circ (\pi|_U)^{-1})$. It therefore follows from the definition of $\mathcal{M}$ that the negative $L^2(\Sigma^\prime)$-gradient $\mathcal{M}^\prime$ of $\mathcal{E}^\prime$ has the form
    $$\left.\mathcal{M}^\prime(v)\right|_x = \left.\mathcal{M}(v\circ\pi^{-1})\right|_{\pi(x)}$$
    for $\pi^{-1}$ a local inverse of $\pi$ near $x$. The new functional $\mathcal{E}^\prime$ satisfies the same ellipticity and analyticity hypotheses as $\mathcal{E}$, and therefore by \cite[Theorem 3]{Sim83a} $\mathcal{E}^\prime$ admits a Łojasiewicz--Simon inequality, i.e. we have for all $v\in C^3(\Sigma^\prime)$ with $|v|_{C^3(\Sigma^\prime)}\leq\sigma^\prime$,
    \begin{equation}\label{E:ls-1}
        |\mathcal{E}^\prime(v)-\mathcal{E}^\prime(0)|^{1-\alpha^\prime/2}\leq\|\mathcal{M}^\prime(v)\|_{L^2(\Sigma^\prime)}
    \end{equation}
    for some fixed $\alpha^\prime = \alpha^\prime(\mathcal{E}^\prime,\Sigma^\prime)\in (0,1]$, $\sigma^\prime = \sigma^\prime(\mathcal{E}^\prime,\Sigma^\prime)>0$. Now, if we define $v\in C^3(\Sigma^\prime)$ by $v(x,u_i(x)) = u_i(x)$, \eqref{E:ls-1} becomes
    \begin{equation}\label{E:ls-2}
        \left|\int_\Sigma\sum^q_{i=1}\left[E(x,u_i(x),\nabla u_i(x)) - E(x,0,0)\right]\ dx\right|^{2-\alpha^\prime} \leq \int_\Sigma\sum^q_{i=1}|\mathcal{M}(u_i)|^2\ dx
    \end{equation}
    which is valid provided $|u|_{C^3(\Sigma)}\leq \sigma^\prime/2$. We highlight that the dependency of \eqref{E:ls-2} on $\Sigma^\prime$ (i.e. on $u$) comes only through the constant $\alpha^\prime,\sigma^\prime$, of which there are only finitely many due to \ref{lemma:covers}. Thus, as there are at most $N(\Sigma,q)$ such degree $q$ covers $\Sigma^\prime$ of $\Sigma$, the result follows by taking $\alpha = \min_{\Sigma^\prime}\alpha(\mathcal{E}^\prime,\Sigma^\prime)\in (0,1]$ and $\sigma = \frac{1}{2}\min_{\Sigma^\prime}\sigma(\mathcal{E}^\prime,\Sigma^\prime)$.
\end{proof}
\begin{remark}
    In the case that $\Sigma$ is simply connected, then the above proof essentially reduces to applying directly \cite[Theorem 3]{Sim83a}. Indeed, in this case the only cover is $\Sigma$ itself, and by Lemma \ref{lemma:ud-properties}(i) we have that $u$ globally decomposes on all of $\Sigma$. Thus, one can just apply \cite[Theorem 3]{Sim83a} to each component $u_i$. Indeed, when $\Sigma$ is simply connected, our proofs essentially reduce to applying \cite{Sim83a} to the single-valued functions in the decomposition.
\end{remark}

\section{Estimates}\label{sec:estimates}

Fix $-\infty<S<T<\infty$. In this section we are interested in $u\in C^3_q(\Sigma\times [S,T])$ solving the PDE
\begin{equation}\label{E:PDE}
    \ddot{u} + \dot{u} + \mathcal{M}(u) + \mathcal{R}(u) = f
\end{equation}
where $f = f(x,t)$ is a $C^1_q(\Sigma\times [S,T])$ function satisfying
\begin{equation}\label{E:f-assumption}
    |f(t)|^*_1 \leq \delta_- e^{-\eps t} + \delta_+ e^{\eps t}
\end{equation}
for some fixed $\eps>0$ and $\delta_{-},\delta_+\in [0,1)$ small (possibly $0$), and $\mathcal{R}(u)$ has the form
\begin{equation}\label{E:R}
    \mathcal{R}(u) = (a_1\nabla^2 u + a_2)\dot{u} + a_3\nabla \dot{u} + a_4\ddot{u}
\end{equation}
where $a_i = a_i(x,u,\nabla u, \dot{u})$ are $C^1$ functions satisfying $a_i(x,0,0,0) \equiv 0$.
Note in particular that \eqref{E:R} implies
$$|\mathcal{R}(u(t))|\leq \Psi(|u(t)|^*_2)\left(|\dot{u}(t)| + |\nabla\dot{u}(t)| + |\ddot{u}(t)|\right),$$
for some $\Psi \in C^0([0, \infty))$ (determined by the $a_i$) satisfying $\Psi(0) = 0$.  Recall that we interpret solving \eqref{E:PDE} as
\begin{equation}\label{E:PDE-2}
    \ddot{u}_i + \dot{u}_i + \mathcal{M}(u_i) + \mathcal{R}(u_i) = f_i \qquad \text{for }i=1,\dotsc,q
\end{equation}
at every point in $\Sigma\times [S,T]$. In particular note that, by the fundamental theorem of calculus and the structure of $\mathcal{M}$ and $\mathcal{R}$, for any $i\neq j$, the difference $w=u_i - u_j$ (locally) solves a PDE of the form
\begin{equation}\label{E:PDE-difference}
    \ddot{w}+\dot{w} + Lw + e_1 \nabla^2 w + e_2\nabla \dot{w} + e_3\ddot{w} + e_4\nabla w + e_5\dot{w} + e_6 w = h
\end{equation}
for $e_k = e_k(x,t)$ and $h$ continuous functions with $|e_k(t)|\leq \tilde \Psi(|u(t)|_2^*)$ for some $\tilde \Psi \in C^0[0, \infty)$ with $\tilde \Psi(0) = 0$.  In particular, provided $|u|^*_2$ is sufficiently small (depending only on $\mathcal{M}$ and $\mathcal{R}$), \eqref{E:PDE-difference} is uniformly elliptic, and thus by standard elliptic theory (see e.g. \cite{NV17}) the singular set of $w$, namely $\{x:(w(x),\nabla w(x)) = (0,0)\}$, has dimension $\leq n-2$.

In particular, this tells us that the coincidence set of $u$ has dimension $\leq n-2$ in $\Sigma\times [S,T]$, and thus by Lemma \ref{lemma:ud} $u$ is uniquely decomposable. By the coarea formula it follows that the coincide set of $u(t)$ in $\Sigma\times\{t\}$ has dimension $\leq n-3$ for a.e. $t>0$, and hence again by Lemma \ref{lemma:ud},
\begin{equation}\label{E:ud-slices}
    u(t) \text{ is uniquely decomposable in $\Sigma\times\{t\}$ for a.e. $t$.}
\end{equation}
From Lemma \ref{lemma:ud-properties}(i), we deduce that there is a radius $r_0 = r_0(\Sigma)$ so that in any region of the form $B_{r_0}(x)\times [\max\{S,t-1\},\min\{T,t+1\}]$, for $x\in \Sigma$ and $t\in [S,T]$, $u$ (uniquely) decomposes into single-valued $C^3$ functions. This decomposition of $u$ and \eqref{E:PDE-2} imply that $f$ decomposes (not necessarily uniquely) into single-valued $C^1$ functions in each region $B_{r_0}(x)\times [\max\{S,t-1\},\min\{T,t+1\}]$ also.

Our main decay estimate is the following. For ease of notation we write (recalling the shorthand $|u|^2 = \sum_i |u_i|^2$)
\begin{equation}\label{E:I-defn}
    I(s,t):= \int^t_s\int_\Sigma |\dot{u}|^2.
\end{equation}
\begin{theorem}\label{thm:decay-main}
    In the notation above, there are constants $\eta = \eta(\Sigma,q,\mathcal{M},\mathcal{R})\in (0,1)$ and $\alpha = \alpha(\Sigma, q)\in (0,1]$ such that the following holds. If $u:\Sigma\times [S,T]\to \R$ solves \eqref{E:PDE}, and
    \begin{equation}
        |u|^*_3\leq\eta, \qquad \max\{\delta_- e^{-\eps S}, \delta_+ e^{\eps T}\}\leq \eta, 
    \end{equation}
    then for all $S+1\leq s<t\leq T-1$, we have the inequality
    \begin{equation}\label{E:thm-decay-main-1}
        I(s,t)^{2-\alpha} \leq C\left[I(s-1,t+1)-I(s,t)\right] + C\delta_-^2 e^{-2\eps s} + C\delta_+^2 e^{2\eps t}.
    \end{equation}
    Consequently, we get the Dini-type estimate
    \begin{align}
        \|\mathcal{G}(u(s),u(t))\|^{2/\alpha}_{L^2(\Sigma)} & \leq C\left(\int^T_S\|\dot{u}(t)\|_{L^2(\Sigma)}\ dt\right)^{2/\alpha}\label{E:thm-decay-main-2}\\
        & \leq C^2\int^T_S\|\dot{u}(t)\|^2_{L^2(\Sigma)}\ dt + C\delta_-^2e^{-2\eps S} + C\delta_+^2 e^{2\eps T}\label{E:thm-decay-main-3}
    \end{align}
    for all $S\leq s<t\leq T$. Here $C = C(\Sigma,q, \eps)>0$.
\end{theorem}

\begin{remark}
For simplicity we write and prove Theorem \ref{thm:decay-main} for $u$ taking (multi-)values in $\R$, however the proof carries over verbatim to $u$ which take values in a general vector bundle over $\Sigma$.
\end{remark}

\begin{proof}
    Fix $\sigma>0$. From the proceeding discussion, provided $\eta = \eta(\mathcal{M},\mathcal{R})>0$ is sufficiently small, there is an $r_0 = r_0(\Sigma)>0$ so that on any domain of the form $B_{r_0}(x)\times [t-\sigma,t+\sigma]$, with $(x,t)\in \Sigma\times [S+\sigma,T-\sigma]$, we can decompose $u = \sum_{i=1}^q\llbracket u_i\rrbracket$ into single-valued $C^3$ functions and $f = \sum^q_{i=1}\llbracket f_i\rrbracket$ into single-valued $C^1$ functions, which solve
    \begin{equation}\label{E:decay-1}
        \ddot{u}_i + \dot{u}_i + \mathcal{M}(u_i) + \mathcal{R}(u_i) = f_i
    \end{equation}
    for each $i=1,\dotsc,q$. First observe that by differentiating \eqref{E:decay-1} in $t$, ensuring that $\eta = \eta(\mathcal{M},\mathcal{R})>0$ is sufficiently small, we get that every $v_i := \dot{u}_i$ satisfies a PDE of the form
    $$\ddot{v}_i + \dot{v}_i + L\dot{v}_i + e_{i,1}\nabla^2 v_i + e_{i,2}\ddot{v}_i + e_{i,3}\nabla \dot{v}_i + e_{i,4}\nabla v_i + e_{i,5}\dot{v}_i + e_{i,6}v_i = g_i,$$
    where $e_{i,j} = e_{i,j}(x,t)$ and $g_i = g_i(x,t)$ are $C^0$ functions on $B_{r_0}(x)\times [t-\sigma,t+\sigma]$ satisfying $|e_{i,j}|\leq\delta$ and $|g_i|\leq \delta_-e^{-\eps t} + \delta_+ e^{\eps t}\leq 2\eta$. Therefore, provided $\eta = \eta(\mathcal{M},\mathcal{R},\sigma)>0$ is sufficiently small, we can use standard (interior) $L^p$ estimates to get
    \begin{equation}\label{E:decay-2}
        \int^{t+\sigma/2}_{t-\sigma/2}\int_\Sigma |\dot{v}|^2 + |\nabla v|^2 \leq C\int^{t+\sigma}_{t-\sigma}\int_\Sigma |v|^2 \qquad \forall t\in [S+\sigma,T-\sigma],
    \end{equation}
    where $C = C(\Sigma,q,\sigma)>0$. In particular, there is a set of times $G_t\subset [t-\sigma/2,t+\sigma/2]$ of measure $|G|\geq \frac{99\sigma}{100}$ so that for every $t^\prime\in G_t$ we have
    \begin{equation}\label{E:decay-3}
    \int_\Sigma |\dot{v}(t^\prime)|^2 + |\nabla v(t^\prime)|^2 + |v(t^\prime)|^2 \leq C\int^{t+\sigma}_{t-\sigma}\int_\Sigma|v|^2
    \end{equation}
where again $C = C(\Sigma,q,\sigma)>0$.
    
Next, observe that if we combine \eqref{E:decay-3} with our Łojasiewicz--Simon inequality (Proposition \ref{prop:ls}), using the fact from \eqref{E:ud-slices}, we have for a.e. $t^\prime \in G_t$ (i.e. those $t^\prime$ for which the uniquely decomposable hypothesis of Proposition \ref{prop:ls} holds):
\begin{align}
    \nonumber |\mathcal{E}_q(u(t^\prime)) - \mathcal{E}_q(0)|^{2-\alpha} & \leq \int_\Sigma \sum_i |\mathcal{M}(u_i(t^\prime))|^2\ dx\\
    \nonumber & \leq C\int_\Sigma |\ddot{u}(t^\prime)|^2 + |\dot{u}(t^\prime)|^2 + |\mathcal{R}(u(t^\prime))|^2 + |f(t^\prime)|^2\ dx\\
    \nonumber & \leq C\int_\Sigma |\ddot{u}(t^\prime)|^2 + |\dot{u}(t^\prime)|^2 + |\nabla\dot{u}(t^\prime)|^2\ dx + C\left(\delta_-^2e^{-2\eps t^\prime} + \delta_+^2 e^{2\eps t^\prime}\right)\\
    & \leq C\int^{t+\sigma}_{t-\sigma}\int_\Sigma |\dot{u}|^2\ dx\, dt + C\left(\delta_-^2e^{-2\eps t^\prime} + \delta_+^2 e^{2\eps t^\prime}\right)\label{E:decay-4}
\end{align}
where again $C = C(\Sigma,q,\sigma)>0$ is a constant which can change from line-to-line; here, $\alpha = \alpha(\Sigma,q)\in (0,1]$ is the constant from Proposition \ref{prop:ls}.

Now take $s<t$ with $s,t\in [S+\sigma,T-\sigma]$ and fix $s^\prime<s$ with $s^\prime \in G_s$ and $t^\prime>t$ with $t\in G_t$. Multiply \eqref{E:decay-1} by $\dot{u}_i$, then sum over $i$, and then integrate over $[s^\prime,t^\prime]$ to obtain
\begin{equation}\label{E:decay-5}
    \int^{t^\prime}_{s^\prime}\int_\Sigma\frac{d}{dt}\left(\frac{1}{2}|\dot{u}|^2\right) + |\dot{u}|^2 + \sum_i \mathcal{M}(u_i)\dot{u}_i + \sum_i \mathcal{R}(u_i)\dot{u}_i\ dt = \int^{t^\prime}_{s^\prime}\int_\Sigma \sum_i f_i \dot{u}_i\ dt.
\end{equation}
Noting that
\begin{equation*}
    \frac{d}{dt}\mathcal{E}_q(u(t)) = -\int_\Sigma\sum_i\mathcal{M}(u_i)\dot{u}_i\ dx
\end{equation*}
we can then compute using \eqref{E:decay-5}, \eqref{E:decay-3}, \eqref{E:decay-2}, recalling that $|s-s^\prime|\leq\sigma/2$, $|t-t^\prime|\leq\sigma/2$, assuming $\eta = \eta(\mathcal{M},\mathcal{R},\sigma)>0$ is sufficiently small,
\begin{align*}
    I(s,t) & \leq I(s^\prime,t^\prime)\\
    & = \int^{t^\prime}_{s^\prime}\int_\Sigma \sum_i f_i \dot{u}_i +\int^{t^\prime}_{s^\prime}\frac{d}{dt}\mathcal{E}_q(u(t)) - \int^{t^\prime}_{s^\prime}\int_\Sigma\frac{d}{dt}\left(\frac{|\dot{u}|^2}{2}\right) - \sum_i\mathcal{R}(u_i)\dot{u}_i\ dt\\
    & \leq \mathcal{E}_q(u(t^\prime)) - \mathcal{E}_q(u(s^\prime)) + \frac{1}{2}\int_\Sigma |\dot{u}(s^\prime)|^2 - \frac{1}{2}\int_\Sigma |\dot{u}(t^\prime)|^2 + \int^{t^\prime}_{s^\prime}\left(f^2 + q^2|\mathcal{R}(u)||\dot{u}|\right) dt + \frac{1}{4}I(s^\prime,t^\prime)\\
    & \leq |\mathcal{E}_q(u(t^\prime))-\mathcal{E}_q(0)| + |\mathcal{E}_q(u(s^\prime))-\mathcal{E}_q(0)| + C\int^{s+\sigma}_{s-\sigma}\int_\Sigma |\dot{u}|^2 + C\int^{t+\sigma}_{t-\sigma}\int_\Sigma|\dot{u}|^2 + \frac{1}{4}I(s^\prime,t^\prime)\\
    & \hspace{2em} + C\cdot \Psi(\eta)\int^{t^\prime}_{s^\prime}\int_\Sigma \left(|\dot{u}|^2 + |\ddot{u}|^2 + |\nabla\dot{u}|^2\right)dt + C\delta_-^2 e^{-2\eps s^\prime} + C\delta_+^2 e^{2\eps t^\prime}\\
    & \leq |\mathcal{E}_q(u(t^\prime))-\mathcal{E}_q(0)| + |\mathcal{E}_q(u(s^\prime))-\mathcal{E}_q(0)| + C\int^{s+\sigma}_{s-\sigma}\int_\Sigma |\dot{u}|^2 + C\int^{t+\sigma}_{t-\sigma}\int_\Sigma|\dot{u}|^2 + \frac{1}{4}I(s^\prime,t^\prime)\\
    & \hspace{2em} + C\cdot \Psi(\eta)\int^{t+\sigma}_{s-\sigma}\int_\Sigma |\dot{u}|^2 dt + C\delta_-^2e^{-2\eps s^\prime} + C\delta_+^2 e^{2\eps t^\prime},
\end{align*}
where here $C = C(\Sigma,q,\sigma, \eps)>0$. Now assuming also $\eta = \eta(\Sigma,q,\mathcal{M},\mathcal{R},\sigma)$ is sufficiently small, splitting the first integral on the (very) last line above as $\int^{t+\sigma}_{s-\sigma} = \int^{t+\sigma}_{t-\sigma} + \int^{t-\sigma}_{s+\sigma} + \int^{s+\sigma}_{s-\sigma}$, we get (continuing the chain of inequalities)
\begin{align*}
    &\leq |\mathcal{E}_q(u(t^\prime))-\mathcal{E}_q(0)| + |\mathcal{E}_q(u(s^\prime))-\mathcal{E}_q(0)| + C\int^{s+\sigma}_{s-\sigma}\int_\Sigma |\dot{u}|^2 + C\int^{t+\sigma}_{t-\sigma}\int_\Sigma|\dot{u}|^2 + \frac{1}{2}I(s,t)\\
    & \hspace{2em} + C\delta_-^2 e^{-2\eps s^\prime} + C\delta_+^2 e^{2\eps t^\prime},
\end{align*}
where we have absorbed the extra parts of $I(s^\prime,t^\prime)$ when compared to $I(s,t)$ into the other terms. Now using \eqref{E:decay-4}, recalling that $\alpha = \alpha(\Sigma,q)\in (0,1]$ (continuing again the chain of inequalities)
\begin{align*}
    &\leq \frac{1}{2}I(s,t) + C\int^{s+\sigma}_{s-\sigma}\int_\Sigma |\dot{u}|^2 + C\int^{t+\sigma}_{t-\sigma}\int_\Sigma|\dot{u}|^2 + C\delta_-^2e^{-2\eps s^\prime} + C\delta_+^2e^{2\eps t^\prime}\\
    & \hspace{2em} + C\left(\int^{s+\sigma}_{s-\sigma}\int_\Sigma |\dot{u}|^2\right)^{\frac{1}{2-\alpha}} + C\left(\int^{t+\sigma}_{t-\sigma}\int_\Sigma |\dot{u}|^2\right)^{\frac{1}{2-\alpha}} + C\left(\delta_-^2e^{-2\eps s^\prime}\right)^{\frac{1}{2-\alpha}} + C\left(\delta_+^2e^{2\eps t^\prime}\right)^{\frac{1}{2-\alpha}}\\
    & \leq \frac{1}{2}I(s,t) + C\left(\int^{s+\sigma}_{s-\sigma}\int_\Sigma |\dot{u}|^2\right)^{\frac{1}{2-\alpha}} + C\left(\int^{t+\sigma}_{t-\sigma}\int_\Sigma |\dot{u}|^2\right)^{\frac{1}{2-\alpha}}\\
    & \hspace{20em} + C\left(\delta_-^2e^{-2\eps s}\right)^{\frac{1}{2-\alpha}} + C\left(\delta_+^2 e^{2\eps t}\right)^{\frac{1}{2-\alpha}}
\end{align*}
as all quantities are $\leq 1$ (we remark that in the ``integrable'' case, i.e. when $\alpha = 1$, the situation is simpler as the various terms have identical powers). Here, $C = C(\Sigma,q,\sigma, \eps)$. Therefore, rearranging we obtain
$$I(s,t)^{2-\alpha} \leq C\left[I(s-\sigma,t+\sigma) - I(s+\sigma,t-\sigma) + \delta_-^2e^{-2\eps s} + \delta_+^2e^{2\eps t}\right].$$
If we now take $\sigma = 1/2$, and replace $s$ by $s-1/2$ and $t$ by $t+1/2$, this gives
$$I(s-1/2,s+1/2)^{2-\alpha} \leq C\left[I(s-1,t+1)-I(s,t) + \delta_-^2e^{-2\eps s} + \delta_+^2e^{2\eps t}\right]$$
where now $C = C(\Sigma,q, \eps)$. Noting that $I(s,t)\leq I(s-1/2,t+1/2)$ then gives \eqref{E:thm-decay-main-1}. 

Now we turn to proving \eqref{E:thm-decay-main-3}. First note that by replacing $t$ with $t-(T+S)/2$ and  $\delta_\pm$ with $\delta_\pm e^{\mp\eps(T+S)/2}$, there is no loss in generality in assuming that $S=-T$. 
Of course, if $T\leq 10$, then \eqref{E:thm-decay-main-3} follows trivially by Hölder's inequality, and so without loss of generality we can assume $T\geq 10$.

Let us write
$$I(t):= I(-t,t) + (\delta_-^2 + \delta_+^2) e^{2\eps t}.$$
Then \eqref{E:thm-decay-main-1} implies (taking $s=-t$)
\begin{equation}\label{E:decay-6}
    I(t)^{2-\alpha} \leq C\left[I(t+1)-I(t)\right] \qquad \text{for all }t\in (0,T-1].
\end{equation}
where $C = C(\Sigma,q, \eps)$.  We split into two cases, depending on whether $\alpha = 1$ or $\alpha\in (0,1)$.

If $\alpha\in (0,1)$, then standard 1-dimensional calculus (as in \cite[Section 3.15(9)]{Sim96} implies
$$I(t)^{\alpha-1}-I(t+1)^{\alpha-1} \geq C_0$$
for some $C_0 = C_0(\Sigma,q, \eps)$. We deduce that, for $t=T-k$, $k\in \N$, we have
$$I(t)\leq \left[I(T)^{\alpha-1}+C_0(T-t)\right]^{\frac{1}{\alpha-1}}.$$
Now define $t_0 = T$, $t_i = T-e^i$, for $i\in \Z_{\geq 1}$ such that $t_i>0$. Let $I$ be the largest integer for which $t_I\geq 0$, and set $t_{I+1} = 0$. We then compute
\begin{align*}
    \int^T_{-T}\|\dot{u}(t)\|_{L^2(\Sigma)} dt & \leq \sum^{I+1}_{i=1}\int_{t_i\leq |t|<t_{i-1}}\|\dot{u}(t)\| dt\\
    & \leq 2\sum^{I+1}_{i=1}|t_{i-1}-t_i|^{1/2}I(t_{i-1})^{1/2}\\
    & \leq 2\cdot e^{1/2}I(T)^{1/2} + 2\sum^{I}_{i=1}e^{i/2}\left[I(T)^{\alpha-1}+ C_0e^i\right]^{\frac{1}{2(\alpha-1)}}\\
    & \leq C(\Sigma,q, \eps)I(T)^{\alpha/2}
\end{align*}
where in the last line we have used that $I(T)<1$ and the elementary inequality
$$\sum^\infty_{i=0}e^{i/2}(A+e^i)^{-\beta/2} \leq C(\beta)A^{(1-\beta)/2} \qquad \text{for }A\geq 0 \text{ and }\beta>1.$$
This establishes \eqref{E:thm-decay-main-3} in the case $\alpha\in (0,1)$.

If $\alpha=1$, we can rearrange \eqref{E:decay-6} to obtain
$$I(t) \leq \frac{C}{C+1}I(t+1) =: A\cdot I(t+1)$$
for some constant $A = A(\Sigma,q)\in (0,1)$. By applying a similar computation to the above except with the (decreasing) sequence $\{t_i\}_{i=1}^I = \{i\in \Z_{\geq 0}: i\leq \lfloor T\rfloor\}$, and $t_0 = T$, we deduce 
$$\int^T_{-T}\|\dot{u}(t)\|_{L^2(\Sigma)} dt \leq 2\cdot 2^{1/2}\sum_{i=0}^\infty A^{i/2}I(T)^{1/2} \leq CI(T)^{1/2}$$
where $C = C(\Sigma,q, \eps)$. This establishes \eqref{E:thm-decay-main-3} in the case $\alpha=1$ also.

The final thing left to show is \eqref{E:thm-decay-main-2}: this essentially follows from the fundamental theorem of calculus and the Minkowski inequality. In any $B_{r_0}(x)\times [\tau-1,\tau+1]$, we can decompose $u = \sum_i \llbracket u_i\rrbracket$ into single-valued $C^3$ functions, and so for $U\subset B_{r_0}(x)$ and $s,t\in [\tau-1,\tau+1]$, we have
\begin{align*}
    \left(\int_U\mathcal{G}(u(s),u(t))^2 dx\right)^{1/2} & \leq \left(\int_U\sum_i|u_i(x,t) - u_i(x,s)|^2 dx\right)^{1/2}\\
    & = \left(\int_U\sum_i\left|\int^t_s\dot{u}_i(x,t^\prime) dt^\prime\right|^2 dx\right)^{1/2}\\
    & \leq \sum_i \left(\int_U\left(\int^t_s\dot{u}_i(x,t^\prime) dt^\prime\right)^2 dx\right)^{1/2}\\
    & \equiv \sum_i \left\|\int^t_s\dot{u}_i(\cdot,t^\prime) dt^\prime\right\|_{L^2(U)}\\
    & \leq \sum_i \int^t_s\|\dot{u}_i(\cdot,t^\prime)\|_{L^2(U)} dt^\prime\\
    & = \sum_i\int^t_s\left(\int_U|\dot{u}_i(x,t^\prime)|^2 dx\right)^{1/2} dt^\prime.
\end{align*}
Now if we take general $S\leq s<t\leq T$, we can choose $s=t_0<t_1<\cdots<t_k=t$ with $|t_{k+1}-t_k|<1$. Then we compute using the above
\begin{align*}
    \|\mathcal{G}(u(s),u(t))\|_{L^2(U)} & \leq \sum^k_{j=1}\|\mathcal{G}(u(t_{j-1}),u(t_j))\|_{L^2(\Sigma)}\\
    & \leq \sum_{j=1}^k\sum_{i=1}^q\int^{t_j}_{t_{j-1}}\|\dot{u}_i\|_{L^2(U)}\\
    & = \sum_i\int^t_s\|\dot{u}_i\|_{L^2(U)}
\end{align*}
Now using that $\|\dot{u}_i\|_{L^2(U)}\leq \|\dot{u}\|_{L^2(U)} \leq \|\dot{u}\|_{L^2(\Sigma)}$, and squaring both sides and summing over a suitable finite number (depending only on $\Sigma$) of balls $U = B_{r_0}(x)$ which cover $\Sigma$, and then taking the square root, we get
$$\|\mathcal{G}(u(s),u(t))\|_{L^2(\Sigma)} \leq C(\Sigma,q)\int^t_s\|\dot{u}\|_{L^2(\Sigma)}$$
which gives \eqref{E:thm-decay-main-2} as desired.
\end{proof}

\section{Proof of Theorem \ref{thm:dini} \& \ref{thm:main-1}}\label{sec:proofs-1}

In this section we fix $\BC$ a minimal hypercone in $\R^{n+1}$ with smooth cross-section $\Sigma$. We also fix a choice of unit normal $\nu_{\BC}$ for $\BC$, which induces a choice of unit normal $\nu_\Sigma\subset S^n$.

Let us introduce some notation. Given an $q$-valued $C^k$ function $u$ on $\BC\cap A_{R,\rho}$, write
$$\graph_{\BC}(u)\cap A_{R,\rho}:= \left\{|x|\cdot\frac{x+u_i(x)\nu_{\BC}(x)}{|x+u_i(x)\nu_{\BC}(x)|}: x\in\BC\cap A_{R,\rho}\text{ and }i=1,\dotsc,q\right\}.$$
Given a $q$-valued $C^k$ function on $\Sigma$, we similarly define
$$\graph_\Sigma(u):= \left\{\frac{\theta + u_i(\theta)\nu_\Sigma(\theta)}{|\theta + u_i(\theta)\nu_\Sigma(\theta)|}:\theta\in \Sigma\text{ and }i=1,\dotsc,q\right\}.$$
First, we prove the a priori estimate in Theorem \ref{thm:dini}.

\begin{proof}[Proof of Theorem \ref{thm:dini}]
    It is fairly standard (c.f. \cite{Sim83a, Sim85}) to show that by changing coordinates via $t = n\log(r)$ and $v(\theta,t) = r^{-1}u(r\theta)$, where $r,\theta$ are polar coordinates on $\R^{n+1}$, that $v$ satisfies a PDE of the form \eqref{E:PDE} considered in Section \ref{sec:estimates}. For the reader's convenience and to highlight the differences in the multi-valued setting, we summarise the ideas here.

    Let us first suppose $q=1$. Given $u\in C^1(\Sigma)$, we define
    $$\mathcal{E}(u):= \frac{1}{n}\H^{n-1}(\graph_\Sigma(u)) \equiv \frac{1}{n} \int_\Sigma E(\theta,u(\theta),\nabla u(\theta)) d\theta.$$
    This is well-defined for $|u|_{C^1}$ sufficiently small. The function $E$ is analytic in its arguments, satisfies the convexity hypothesis \eqref{E:convexity}, and has the form
\begin{equation}\label{eqn:dini-.1}
E(\theta,z,p)^2 = 1+\frac{1}{2}(|p|^2 - z^2|A_\Sigma(\theta)|^2 - z^2 (n-1)) + z^3 \cdot E_1 + z \cdot p^2 \cdot E_2 + p^3 \cdot E_3, 
\end{equation}
for $E_1, E_2, E_3$ analytic functions of $(\theta, z, p)$.  In both \eqref{eqn:dini-.1}, \eqref{eqn:dini-.2} we are abusing notation slightly: when we write $p^2 \cdot E_2$ we really mean there are functions $E_{2ij}^{\alpha \beta}(\theta, z, p)$ defined on $\Sigma \times \R^{n+1} \times (\R^{n+1})^2$, so that
\[
p^2 \cdot E_2 \equiv \sum_{i, j,  \alpha, \beta} p_i^\alpha p_j^\beta E_{2ij}^{\alpha \beta}.
\]
    
    If $M\cap A_{R,\rho} = \graph_{\BC}(u)\cap A_{R,\rho}$ for $u\in C^3(\BC\cap A_{R,\rho})$, then by setting $t = n\log(r)$ and $v(\theta,t) = r^{-1}u(r\theta)$, $v$ becomes a function on $\Sigma\times [n \log(\rho), n \log(R)] =: \Sigma\times [S,T]$, and we can write
    $$\H^n(M\cap A_{R,\rho}) = \int^T_S\int_\Sigma e^t\left[F(\theta,v,\nabla v,\dot{v}) + H(\theta,g(\theta,t,v),v,\nabla v,\dot{v})\right]\, d\theta dt.$$
    Here $F,H$ are analytic functions of their arguments, with $F(\theta,v,\nabla v,0) \equiv E(\theta,v,\nabla v)$, $H(\theta,0,v,\nabla v,\dot{v})\equiv 0$, and with the additional structure 
    \begin{equation}\label{eqn:dini-.2}
    F(\theta,z,p,q)^2 = [1+|q|^2 +q^2\cdot z^2\cdot F_1(z)]\cdot E(\theta,z,p)^2 + q^2\cdot p^2\cdot F_2(\theta,z,p).\end{equation}

    A direct computation shows that, provided $|v|^*_3$ is sufficiently small, then $M$ is stationary in $A_{R,\rho}$ if and only if $v$ solves a PDE of the form
    \begin{equation}\label{E:PDE-11}
        \ddot{v}+\dot{v}+\mathcal{M}(v) + \mathcal{R}(v) = f
    \end{equation}
    where $\mathcal{M}$ is the negative $L^2(\Sigma)$-gradient of $\mathcal{E}$ (as in Section \ref{sec:prelim}), $\mathcal{R}$ has the form \eqref{E:R}, and $f$ is an analytic function of $(\theta,g(\theta,t,v), Dg(\theta,t,v), v,\nabla v,\dot{v},\ddot{v},\nabla\dot{v},\nabla^2 v)$, which satisfies
    $$f(g=0, Dg=0) \equiv 0.$$
    In particular, from \eqref{E:Dini-1} $f$ is a $C^1$ function in $(\theta,t)$, satisfying
    \begin{equation}\label{E:f-bound-2}
        |f|_1^*\leq c(\Sigma)\delta_- e^{-\beta t} + c(\Sigma)\delta_+ e^{\beta t}
    \end{equation}
    which is exactly the set-up considered in Section \ref{sec:estimates}.

    If instead $M\cap A_{R,\rho} = \graph(\BC)(u)\cap A_{R,\rho}$ for $u\in C^3_q(\BC\cap A_{R,\rho})$ (so, $u$ is a non-branched $q$-valued function), then by effecting the same change of coordinates and setting $v = r^{-1}u(r\theta)$, we obtain $v\in C^3_q(\Sigma\times [S,T])$, and now
    $$\H^n(M\cap A_{R,\rho}) = \int^T_S\int_\Sigma \sum^q_{i=1}e^t[F(\theta,v_i,\nabla v_i,\dot{v}_i) + H(\theta,g(\theta,t,v_i),v_i,\nabla v_i,\dot{v}_i)]\, d\theta dt,$$
    which is well-defined due to the definition of $C^3_q$ and \eqref{E:ud-slices}. It then follows that $M$ is stationary in $A_{R,\rho}$ as an immersion if and only if $v$ solves the PDE \eqref{E:PDE-11} in the sense of Sections \ref{sec:prelim} and \ref{sec:estimates}, with $f\in C^1_q(\Sigma\times [S,T])$ satisfying the same estimate \eqref{E:f-bound-2}, except with $C = C(\Sigma,q)$.

    In particular, we can apply Theorem \ref{thm:decay-main} to deduce the estimate
    \begin{equation}\label{E:Dini-5}
        \left(\int^R_\rho\left\|r\frac{\del}{\del r}\left(\frac{u(r\cdot)}{r}\right)\right\|_{L^2(\Sigma)} \frac{dr}{r}\right)^{2/\alpha} \leq C\int^R_\rho \left\|r\frac{\del}{\del r}\left(\frac{u(r\cdot)}{r}\right)\right\|_{L^2(\Sigma)}^2 \frac{dr}{r} + \frac{\delta_-^2\rho^{-2\beta}}{\beta e^\beta} + \frac{\delta_+^2 R^{2\beta}}{\beta e^\beta}
\end{equation}
    where $\alpha = \alpha(\Sigma,q)\in (0,1]$ and $C = C(\Sigma,q)>0$. The claim \eqref{E:Dini-3} then follows from \eqref{E:Dini-5} and the Hardt--Simon-type inequality (c.f. \cite[Page 561]{Sim83a}): if
    $$X = r\frac{\theta+r^{-1}u_i(r\theta)\nu_{\BC}(\theta)}{|\theta+r^{-1}u(r\theta)\nu(\theta)|}$$
    then provided $\eps = \eps(\BC)\in (0,1)$ is sufficiently small, we have
    $$\frac{1}{2}\left|r\frac{\del}{\del r}\left(\frac{u_i(r\theta)}{r}\right)\right|^2 \leq \frac{|X^{T_X M}|^2}{|X|^2}\leq 2\left|r\frac{\del}{\del r}\left(\frac{u_i(r\theta)}{r}\right)\right|^2.$$
    Moreover, \eqref{E:Dini-2} follows from \eqref{E:thm-decay-main-2}. Finally, \eqref{E:Dini-4} follows from \eqref{E:Dini-3} and the monotonicity formula. This completes the proof.
\end{proof}
We now obtain Theorem \ref{thm:main-1} by combining Theorem \ref{thm:dini} with the sheeting theorem of Bellettini \cite[Theorem 6]{Bel23}.

\begin{proof}[Proof of Theorem \ref{thm:main-1}]
    We first claim the following: given $\eps>0$ and $q\in \Z_{\geq 1}$, there is a $\delta = \delta(\BC,q,\eps)>0$ such that if $M$ is a stable immersed minimal hypersurface in $A_{8,1/16}$ with $\mathcal{H}^{n-2}(\sing(M))=0$ satisfying
    \begin{equation}\label{E:thm-1-1}
        M\cap A_{1,1/2} = \graph_{\BC}(u), \qquad |u|_{C^2}\leq \delta
    \end{equation}
    for some $u\in C^2_q(\BC\cap A_{1,1/2})$, and
    \begin{equation}\label{E:thm-1-2}
        \H^n(M\cap B_{8,1/16})\leq (q+1/2)\H^n(\BC\cap A_{8,1/16}), \qquad \int_{M\cap A_{8,1/16}}\frac{|\pi^\perp(x)|^2}{|x|^{n+2}} d\H^n(x)\leq \delta,
    \end{equation}
    then $u$ can be extended to a $C^2$ $q$-valued function on $\BC\cap A_{4,1/4}$, with
    \begin{equation}\label{E:thm-1-3}
        M\cap A_{4,1/8} = \graph_{\BC}(u)\cap A_{4,1/8}, \qquad |u|_{C^3}\leq \eps.
    \end{equation}
    To see this, we argue by contradiction. So suppose that there are sequences $\delta_i\to 0$ and $M_i$ immersed stable minimal hypersurfaces in $A_{8,1/16}$ satisfying \eqref{E:thm-1-1}, \eqref{E:thm-1-2} with $\delta_i$, $M_i$ in place of $\delta$, $M$, but for which \eqref{E:thm-1-3} fails. After passing to a subsequence, \eqref{E:thm-1-2} implies that we can take a limit $M_i\weakly V$ as varifolds to obtain some stationary, dilation-invariant, integral varifold $V$ in $A_{6,1/12}$. Then \eqref{E:thm-1-1} implies that $V\res A_{1,1/2} = q|\BC|\res A_{1,1/2}$, and so $V = q|\BC|$ on all of $A_{6,1/12}$. The Bellettini's $\eps$-regularity theorem \cite[Theorem 6]{Bel23} then implies that $M_i\to q|\BC|$ as $C^2$ $q$-valued graphs on compact subsets of $A_{5,1/10}$, which implies that \eqref{E:thm-1-3} holds for all $i$ sufficiently large, providing the desired contradiction. This proves the claim.

    Let us now prove the first case of the theorem. So suppose $M$ is a stable immersed minimal hypersurface in $B_1\subset \R^{n+1}$ and that there is some sequence $r_i\to 0$ so that the rescaled surfaces $r_i^{-1}M$ obey $r_i^{-1}M\weakly q|\BC|$ as varifolds for some regular cone $\BC$. Let $\alpha = \alpha(\BC,q)\in (0,1]$, $\eps = \eps(\BC,q)\in (0,1)$ be as in Theorem \ref{thm:dini}. Also assume that $\delta = \delta(\BC,q,\eps)\in (0,1)$ is sufficiently small to apply the above claim. Fix $i$ sufficiently large so that (by the monotonicity formula)
    \begin{equation}\label{E:thm-1-4}
        \Theta_M(r_i) - \Theta_M(0) \leq \delta^{4/\alpha},
    \end{equation}
    and (by \cite[Theorem 6]{Bel23} and subsequent higher order interior estimates which follow from compactness of $\Sigma$)
    \begin{equation}\label{E:thm-1-5}
        r^{-1}_i M\cap A_{1,1/64} = \graph_{\BC}(u) \qquad \text{for some }u\in C^3_q(\BC\cap A_{1,1/64}) \text{ obeying }|u|_{C^3}<\delta^2.
    \end{equation}
    Now let $\rho^*$ be the infimum over radii $\rho$ for which
    \begin{equation}\label{E:thm-1-6}
        r_i^{-1}M\cap A_{1,\rho} = \graph_{\BC}(u), \qquad \sum^3_{k=0}|x|^{k-1}|\nabla^k u| \leq \eps
    \end{equation}
    for some $u\in C^3_q(\BC\cap A_{1,\rho})$. We claim that $\rho^*=0$. If not, i.e. $\rho^*>0$, then \eqref{E:thm-1-6} fails at $\rho = \rho^*$ and is true for $\rho>\rho^*$. Note first that $r_i^{-1}M$ is still graphical over $A_{1,\rho^*/2}$ due to the estimate in \eqref{E:thm-1-6} holding for $\rho>\rho^*$ (which therefore allows for extension), and so it must be the estimate in \eqref{E:thm-1-6} which fails. But then, combining \eqref{E:Dini-2}, \eqref{E:Dini-4}, with \eqref{E:thm-1-4}  we get
    $$\sup_{r,s\in (\rho^*/2,1)}\|\mathcal{G}(r^{-1}u(r),s^{-1}u(s))\| \leq C\delta^2$$
    and then combining this with \eqref{E:thm-1-5} and the triangle inequality we get
    $$\sup_{r\in (\rho^*/2,1)}\|r^{-1}u(r\theta)\|_{L^2(\Sigma)} \leq (1 + C)\delta^2$$
    where $C = C(\BC,q)$ is an absolute constant independent of $\rho^*$. In particular, applying the standard interior estimate (c.f. the proof of Theorem \ref{thm:decay-main}) we have
    \begin{equation}\label{E:thm-1-7}
    \sum^3_{k=0}|x|^{k-1}|\nabla^k u|\leq C\delta^2 \qquad \text{on }\BC\cap A_{1,\rho^*}
    \end{equation}
    where $C = C(\BC,q)$. So, provided $\delta = \delta(\BC,q,\eps)$ is sufficiently small, we can use \eqref{E:thm-1-7} with \eqref{E:thm-1-4} and our initial Claim to arrive at a contradiction. Hence, $\rho^* = 0$.

    Since $\rho^*=0$, we can use \eqref{E:Dini-4} to deduce that $r^{-1}u(r\cdot)$ converges in $L^2(\Sigma)$ as $r\to 0$ to some unique limit $\phi = \phi(\theta)$. By elliptic estimates, this convergence is in fact in $C^2(\Sigma)$. Our hypothesis that $r^{-1}_iM\weakly q|\BC|$ then implies that $\phi = 0$, and hence $r^{-1}M\to q|\BC|$ as $r\to 0$, giving uniqueness of the tangent cone. To get the rate of decay, we can apply Theorem \ref{thm:decay-main} (i.e. \eqref{E:thm-decay-main-1}) with $s=-\infty$ to deduce (in this situation $\delta_- = \delta_+ = 0$)
    \begin{equation}\label{E:thm-1-8}
    I(\rho)^{2-\alpha} \leq C(I(2\rho) - I(\rho)) \qquad \text{for all }\rho\in (0,1/2)
    \end{equation}
    where $I(\rho) := \int_{M\cap B_\rho}|x|^{-n-2}|\pi^\perp_M(x)|^2 dx$ and $C = C(\BC,q)$. The argument for the decay rate now proceeds as in \cite[Section 3.15]{Sim96}, completing the proof of Theorem \ref{thm:main-1} in this case.

    Now let us turn to the second situation in Theorem \ref{thm:main-1}, namely when $M\subset \R^{n+1}\setminus B_1$ and $r_i\to \infty$. We can apply the same reason as above, except with the monotonicity of \cite{MN22}: without loss of generality, assuming that $M\cap \del B_2$ is an immersed surface, if one defines
    $$\widetilde{\Theta}_M(r):= \frac{\H^n(M\cap A_{r,2})}{\omega_n r^n} + \frac{2}{n\omega_n r^n}\int_{M\cap \del B_2}\frac{|\pi_M(x)|}{|x|}\, d\H^{n-1}$$
    then (the computations of) \cite[Theorem 2.7]{MN22} gives the monotonicity formula
    $$\widetilde{\Theta}_M(r)-\widetilde{\Theta}_M(s)\geq\frac{1}{\omega_n}\int_{M\cap A_{r,s}}\frac{|x^\perp|^2}{|x|^{n+2}}\, d\H^n \qquad \text{for all }2<s<r<\infty.$$
    The argument then proceeds in the same manner as the above, with $\widetilde{\Theta}_M$ in place of $\Theta_M$, and instead of \eqref{E:thm-1-8}, in the application of \ref{E:thm-decay-main-1} we instead take $t=\infty$ and get
    $$J(\rho)^{2-\alpha}\leq C[J(\rho/2)-J(\rho)] \qquad \text{for all }\rho\in (2,\infty)$$
    where $J(\rho) := \int_{M\setminus B_\rho}|x|^{-n-2}|\pi^\perp_M(x)|^2\, dx$. This completes the proof of Theorem \ref{thm:main-1}.
\end{proof}

\section{Proof of Theorem \ref{thm:main-2}}\label{sec:proofs-2}

\begin{proof}[Proof of Theorem \ref{thm:main-2}]
    Let $\Sigma\hookrightarrow \R^{n+1}$ be the immersion of the link into $\R^{n+1}$ ($\Sigma$ could either connected or disconnected). Since one has the $\eps$-regularity theorem \cite[Theorem D]{MW23}, one can lift the functions $u$ defined on (slices of) regions of $\BC$ to those over $\Sigma$. Consequently, we get a Łojasiewicz--Simon inequality in this setting as an immediate consequence of \cite[Theorem 3]{Sim83a} (without any need to use multi-valued functions). The proof now proceeds in an identical way to \cite{Sim83a}.
\end{proof}

\section{Proof of Theorem \ref{thm:main-3}}\label{sec:proofs-3}

The idea is to apply the general existence methods of \cite{CHS, Sim85} to the stationarity equation lifted to an appropriate finite cover of $\Sigma$, and construct solutions that do not descend to embedded graphs over $\Sigma$. Write $A = A(\Sigma)\subset \Z_{\geq 2}$ for the set of possible finite degrees $\geq 2$ of covering spaces of $\Sigma$; since $\Sigma$ is not simply connected, it follows that $\#A \geq 1$.

Fix a covering space $\pi:\Sigma^\prime\to\Sigma$ of $\Sigma$ of degree $q\in A$ (so $q\geq 2$), and endow $\Sigma'$ with the pullback metric to make $\pi$ a local isometry.  Given $v \in C^1(\Sigma' \times [-T, 0])$, we can define the multi-valued function $u \in C^1_q(\BC \cap A_{1, e^{-T/n}})$ by 
\begin{equation}\label{eqn:ex-proof-1}
u(x = r\theta) = \sum_{\theta' \in \pi^{-1}(\theta)} \llbracket r v(\theta', n \log(r)) \rrbracket.
\end{equation}
We are effectively performing the same change of coordinates $t = n \log(r)$, $v = u(x)/|x|$ as in Section \ref{sec:proofs-1}, except we are also unwinding the multi-valued function $u$ on $\Sigma$ into the single-valued function $v$ on $\Sigma'$ (or vice versa).

By the same computations as in Section \ref{sec:proofs-1}, we have
\begin{gather*}
\H^{n-1}(\graph_\Sigma(u)) = \int_{\Sigma'} E(\theta, v, \nabla v) d\theta, \\
\H^n(\graph_\BC(u) \cap A_{1, e^{-T/n}}) = \int_0^T \int_{\Sigma'} e^t F(\theta, v, \nabla v, \dot v) d\theta' dt,
\end{gather*}
where $E, F$ satisfy the same structural conditions \eqref{eqn:dini-.1}, \eqref{eqn:dini-.2}.  In particular, $\graph_\BC(u)$ is stationary for the area functional $\iff$ $v$ solves a quasi-linear equation of the form
\begin{equation}\label{eqn:ex-proof-2}
\ddot v + \dot v + L' v + a_1' \nabla^2 v + a_2' \nabla \dot v + a_3' \ddot v + a_4'\nabla v + a_5' \dot v + a_6' v = 0,
\end{equation}
where each $a_i'$ is a smooth function of $(\theta, v, \dot v, \nabla v)$ satisfying $a_i'(\theta, 0, 0, 0) = 0$, and $L'$ is an $L^2(\Sigma')$-self-adjoint elliptic linear operator on $\Sigma'$.  In fact $L' |_{\theta} = \frac{1}{n^2} (\Delta_\Sigma + |A_\Sigma|^2 + (n-1) )|_{\pi(\theta)}$.

Following \cite{CHS}, we note that $L'$ has an $L^2(\Sigma')$-orthonormal basis of eigenfunctions $\{ \phi_j \in C^\infty(\Sigma') \}_j$ with corresponding eigenvalues $\lambda_j \to \infty$.  Any solution $w$ to the homogenous problem
\[
\ddot w + \dot w + L' w = 0
\]
can be therefore be expanded as
\[
w = \sum_{\lambda_j < -1/4} \left[ c_j \cos(\mu_j t) + d_j \sin(\mu_j t) \right] e^{-t/2} \phi_j + \sum_{\lambda_j = -1/4} \left[ c_j + d_j t \right] e^{-t/2} \phi_j + \sum_{\lambda_j > -1/4} \left[ c_j e^{\gamma^+_j t} + d_j e^{\gamma^-_j t} \right] \phi_j
\]
where $c_j, d_j$ are constants, $\gamma_j^\pm = -1/2 \pm \sqrt{1/4 + \lambda_j}$, and $\mu_j = \mathrm{Im}(\gamma^+_j)$ .  Since $\lambda_1 < 0$ (plug in $1$ into the Raleigh quotient and use the explicit form of $L'$ above), $\gamma^\pm_1 < 0$.

Fix a number $m \in (0, \infty)$ so that $m \neq \gamma^\pm_j$ for any $j$, then fix $J \in \N$ so that $\gamma^+_J < m < \gamma^+_{J+1}$.  Define the projection operator $\Pi_J : L^2(\Sigma') \to L^2(\Sigma')$ by
\[
\Pi_J(f) = f - \sum_{j=1}^J \langle f, \phi_j\rangle_{L^2(\Sigma')} \phi_j, 
\]
and note $\Pi_J$ extends to a linear operator $C^{2, \alpha}(\Sigma') \to C^{2,\alpha}(\Sigma')$ for any fixed $\alpha \in (0, 1)$.  Finally, for $v \in  \Sigma' \times (-\infty, 0)$, define the norm
\[
|v|_B = \sup_{t \in (-\infty, 0)} e^{-mt} |v(t)|_2^*.
\]
Then \cite{CHS, Sim85} prove the following existence theorem:
\begin{theorem}\label{thm:chs}
In the notation above, there are $\eps = \eps(\Sigma^\prime,m,\alpha)>0$ and $c = c(\Sigma^\prime,m,\alpha)>0$ so that given any $g \in C^{2,\alpha}(\Sigma')$ with $|g|_{C^{2,\alpha}} \leq \eps$, one can find $v \in C^{2,\alpha}_{\textnormal{loc}}( (-\infty, 0] \times \Sigma')$ solving \eqref{eqn:ex-proof-2} on $(-\infty, 0) \times \Sigma'$, having the boundary data $\Pi_J(v) = \Pi_J(g)$ on $\{0\}\times \Sigma'$, and satisfying the estimates
\[
|v|_B \leq c |g|_{C^{2,\alpha}}, \quad |v - H_J(g)|_B \leq c |g|_{C^{2,\alpha}}^2,
\]
where
\[
H_J(g) = \sum_{j \geq J+1} \langle g, \phi_j\rangle_{L^2(\Sigma')} e^{\gamma_j^+ t} \phi_j.
\]
\end{theorem}

To prove Theorem \ref{thm:main-3}, we can simply apply Theorem \ref{thm:chs} to $g = \eps' \phi_j$ for $\phi_j$ ($j > J$) any choice of eigenfunction that \emph{does not} descend to a single-valued function on $\Sigma$.  For each sufficiently small $\eps'$ we obtain a solution $v_{\eps'}$ satisfying
\begin{equation}\label{eqn:ex-proof-3}
|v_{\eps'} - \eps' \phi_j e^{\gamma_j^+ t}|_B \leq c \eps'^2.
\end{equation}
For $\eps'$ sufficiently small \eqref{eqn:ex-proof-3} will guarantee $v_{\eps'}$ does not descend to a single-valued function on $\BC \cap A_{1, 0}$ either, i.e. if we define $u_{\eps'} \in C^2_q(\Sigma)$ from $v_{\eps'}$ like in \eqref{eqn:ex-proof-1} then the $u_{\eps'}$ will \emph{not} be single-valued.  Therefore $\graph_\BC(u_{\eps'}) \cap A_{1, 0}$ will be a genuinely immersed but not embedded smooth minimal surface in $A_{1, 0}$, with tangent cone at $0$ being $q|\BC|$.

We note such a $\phi_j$ always exists by the Weyl law for Schrödinger operators (see e.g. \cite[Corollary 1.2]{ChRo}): if every $\phi_j$ for $j \gg 1$ descended to a single-valued eigenfunction on $\Sigma$, then the spectrum with multiplicity of $L'$ on $\Sigma'$ would asymptotically be the same as the spectrum of $L$ on $\Sigma$, contradicting the fact that the volume of $\Sigma'$ is at least twice the volume of $\Sigma$.

To verify strict stability of $\graph_\BC(u_{\eps'})$, we first note that the first eigenfunction $\phi_1$ of $L'$, being multiplicity-one, necessarily descends to the first eigenfunction of $L = \frac{1}{n^2}(\Delta_\Sigma + |A_\Sigma|^2 + (n-1))$ on $\Sigma$.  So if $\BC$ is strictly stable, then $\lambda_1 > -1/4$, and the argument to show strict stability of $\graph_\BC(u_{\eps'})$ proceeds as in \cite{CHS}. \qed

\bibliographystyle{amsalpha}
\bibliography{references.bib}
\end{document}